%% file: Harmoniques_sep2015.tex
\documentclass[12pt,reqno]{article}

\usepackage{amsmath,amsthm,amsfonts,amssymb,amscd}

\def\ZZ {{\mathbb Z}}
\def\NN {{\mathbb N}}

\def\RR {{\mathbb R}}
\def\CC {{\mathbb C}}

\def\PP {{\mathbb P}}

\def\Si{\Sigma}

\def\De{\Delta}

  \def\cG{{\cal G}} \def\cM{{\cal M}} 
\def\cB{{\cal B}}    
    
   \def\cP{{\cal P}} 
\def\cE{{\cal E}}    \def\cW{{\cal
W}}
\def\cF{{\cal F}}  \def\cL{{\cal L}}  \def\cX{{\cal
X}}
    \def\cW{{\cal W}}

\newtheorem{theo}{Theorem}

\newtheorem{theor}{Theorem}[section]
\newtheorem{lemm}[theor]{Lemma}

\newtheorem{coro}[theor]{Corollary}
\newtheorem{corol}{Corollary}[]

\newtheorem{prop}[theor]{Proposition}

\newtheorem{rema}[theor]{Remark}
\newtheorem{defi}[theor]{Definition}

\renewcommand{\div}{\operatorname{div}}

\newtheorem{exam}{Example}

\newenvironment{demo}{\smallskip \noindent{\bf Proof: }}{\hfill$\Box$\medskip}
\newenvironment{dem}{\smallskip \noindent{\bf Proof: }}{\hfill$\Box$\medskip}
\newenvironment{demt}{\smallskip \noindent{\bf Proof of Theorem 1: }} {\hfill$\square$\medskip}
\newenvironment{demt2}{\smallskip \noindent{\bf Proof of Theorem 2: }} {\hfill$\square$\medskip}
\newenvironment{demt3}{\smallskip \noindent{\bf Proof of Theorem 3: }} {\hfill$\square$\medskip}

\newenvironment{demt4}{\smallskip \noindent{\bf Proof of Theorem 4: }} {\hfill$\square$\medskip}

\title{ Foliated Hyperbolicity and Foliations with Hyperbolic Leaves}
\author{Christian Bonatti\thanks{Partially supported by CNRS and LAISLA.} \and  Xavier G\'omez-Mont\thanks{Partially supported by CONACYT 134081 and LAISLA.} \and Matilde Mart\'{\i}nez\thanks{Partially supported by CIMAT and ANII-FCE 2007-106.}}

\begin{document}

\maketitle

\begin{abstract} Given a lamination in a compact space and a laminated vector field $X$ which is hyperbolic when
restricted to the leaves of the lamination,
 we   distinguish a class of $X$-invariant probabilities
that describe the behaviour of   almost every $X$-orbit in every leaf, that we call   u-Gibbs states.
We apply this to the case of foliations in compact manifolds having leaves with negative curvature,
using the foliated hyperbolic vector field on the unit tangent bundle to the foliation generating the leaf geodesics.
When the Lyapunov exponents of such an ergodic u-Gibbs states are negative, it is an SRB-measure
(having a positive Lebesgue basin of attraction).
When the foliation is by hyperbolic leaves, this class of probabilities coincide with the classical
harmonic measures introduced by L. Garnett.
If furthermore the foliation is transversally conformal and does not admit a
transverse invariant measure we show that  the ergodic u-Gibbs states are finitely many, supported
each in one minimal set  of the foliation,  have negative Lyapunov exponents and the union of their basins of attraction has
full Lebesgue measure. The leaf geodesics emanating from a point have a proportion
 whose asymptotic statistics is described by each of these ergodic u-Gibbs states, giving rise to continuous visibility functions of the attractors.
Reversing time, by considering $-X$, we obtain the existence of the same number of repellors of the foliated geodesic flow having the same
harmonic measures as projections to $M$. In the case of only 1 attractor, we obtain a North to South pole dynamics.
\end{abstract}

\input{Sect1_sep2015}

\input{Sect2_sep2015}

\input{Sect3_sep2015}

\input{Sect4_sep2015}

\input{Sect5_sep2015}

\input{Sect6_sep2015}

\input{Sect7_sep2015}
\bibliography{Bibliography}{}
\bibliographystyle{plain}

\end{document}

%% file: Sect1_sep2015.tex
\section{Introduction}

\subsection{Foliated hyperbolicity versus partial hyperbolicity}

Hyperbolicity is a fundamental tool for understanding chaotic behavior of dynamical systems,
even for non hyperbolic systems. The aim of this paper is to show  that ideas from
partially hyperbolic dynamics may be useful in the study of foliations, in particular when the leaves
carry a negatively curved, or better a hyperbolic, Riemannian metric. The foliated geodesic flow is then hyperbolic
along the leaves. However, we make no   assumption on the normal direction of the foliation,
so that the foliated geodesic flow does not need to be partially hyperbolic: the directions along the leaves
are hyperbolic, the normal directions may be anything.

This behavior resembles the notion of \emph{partially hyperbolic diffeomorphism} $f$,
where the tangent bundle of the whole manifold splits in a sum of three $Df$-invariant bundles
$E^s\oplus E^c\oplus E^u$, the stable and unstable bundles being uniformly contracted and expanded,
respectively. In the center bundle, no expansion nor contraction are required: the only requirement
is that the center bundle is dominated by the stable and the unstable ones, meaning that no contraction
in the center bundle is stronger than the contraction in the stable one and no expansion is
stronger than the expansion in the unstable bundle. Partial hyperbolicity implies the existence
of stable and unstable foliations, tangent to $E^s$ and $E^u$, respectively. Furthermore
a fundamental tool for the ergodic theory of partially hyperbolic dynamics is the fact that these foliations
are absolutely continuous with respect to Lebesgue, assuming that the diffeomorphism $f$ is $C^2$,
and distinguishing certain probabilities, called u-Gibbs states.

The foliated geodesic flow of a foliation by  leaves with negative curvature is not necessarily
partially hyperbolic: the hyperbolic behavior in the leaves does not dominate the normal behavior,
that is, the holonomy along the geodesics. However, this lack of domination is  substituted
by the \emph{ a priori} existence of a foliation carrying the flow, allowing to recover the existence
of the stable and unstable foliations tangent to the hyperbolic stable and unstable manifolds,
their absolute continuity with respect to Lebesgue, as
well as distinguishing the corresponding  u-Gibbs states . Thus we will define the notion of \emph{foliated hyperbolicity}.

\subsection{Foliated hyperbolicity}

In Section 2  we consider a lamination $\cF$ in the compact space $M$ and an
\emph{$\cF$-hyperbolic vector field}
$X$: a $C^r$, $r\geq 1$, vector field  in each leaf that varies continuously for the $C^r$ topology  and
is uniformly hyperbolic when restricted to any leaf of $\cF$ (Definition~\ref{d.F-hyp}).
The behavior of an $\cF$-hyperbolic vector field is very close to the one of a continuous
family of Anosov vector fields, and Section 2 states classical properties of the hyperbolic dynamics which can
be straightfowardly generalized to $\cF$-hyperbolic vector fields.  Thus the leaves of $\cF$ are foliated by  stable and unstable
foliations ($\cW^s_{\cF,X}$, $\cW^u_{\cF,X}$)  whose leaves are
$C^r$ and vary continuously for the $C^r$ topology   (Theorem~\ref{t.invariant-manifolds}).
   The central unstable foliation
$\cW^{cu}_{\cF,X}$ is generated by $\cW^u_{\cF,X}$ and $X$, its leaves are contained in the leaves of
$\cF$ and since for backward time $X_t$ is contracting,  it is formed of
$X$-orbits which share the same asymptotic past.

For $X$ an $\cF$-hyperbolic vector field,   we introduce in Definition $3.5$ the $u$-Gibbs states of $(\cF,X)$,
which  are $X$-invariant probability measures on $M$ which are absolutely continuous along the leaves
of the unstable foliation of $X$ in $\cF$ and have a prescribed positive density in each unstable leaf
(coming from the distant past).
In a similar way to the partially hyperbolic setting (see \cite[Theorem 11.16]{BDV}),  we prove in Section 3
that the statistical behavior of most of the orbits of $X$ is described by  the $u$-Gibbs states:

\begin{theo}\label{uGibb2}
Let $(M,\cF)$ be a compact metrizable space with a $C^r$-lamination and $X$ an $\cF$-hyperbolic vector field. Then
there exists a set $E \subset M$ intersecting every leaf of the center unstable foliation $\cW^{cu}_{\cF,X}$ of $X$ in $\cF$
on a full Lebesgue measure subset,
such that for any $ p \in E$, every accumulation measure $\nu$ of
$$\nu_{T, p}:=X_*(\cdot,p) \vert_{[0,T]} \frac{dt}{T}$$
with $T \longrightarrow \infty$ is a u-Gibbs state.
\end{theo}

If furthermore $\cF$ is a  $C^1$ foliation (that is, if the transverse structure of $\cF$ is $C^1$ manifold),
 the linear holonomy cocycle over the flow of $X$ is well
defined on the normal bundle to $\cF$.
 Oseledets' theorem can then be applied to every $X$-invariant ergodic probability, defining \emph{the normal
 Lyapunov exponents of the measure}. They describe the exponential rate of growth of the
holonomy of the foliation along most of the $X$-orbits contained in the
support of the measure.

When the vector field is not just $C^2$ in the leaves, but $C^2$ in the ambient manifold, Pesin theory can be applied to fully
describe the ergodic attractors of the flow. Recall that an $X$-invariant measure $\mu$ is an SRB-measure if its basin of attraction has positive Lebesgue measure. In section 5 we prove:

\begin{theo}
\label{main_theorem}
Let $M$ be a compact manifold with a  smooth foliation $\cF$ and $X$ an  $\cF$-hyperbolic vector field which is $C^2$ on $M$, then:

1) If $\mu$  is an ergodic u-Gibbs state for $X$
with normal Lyapunov exponents which are all negative, then $\mu$ is an SRB-measure.

2) If for every  ergodic u-Gibbs state for $X$
the normal Lyapunov exponents are all negative, then  there are only a finite number of ergodic u-Gibbs states, each being an SRB-measure
and the union of their basins of attraction has full Lebesgue measure.
\end{theo}

\subsection{Laminations by negatively curved leaves}
If one considers laminations $\cF$ by leaves with a $C^3$-Riemannian metric along the leaves of $\cF$ with leafwise strictly negative curvature, the laminated geodesic flow is a flow in the unit tangent bundle $T^1\cF$ of the lamination, which is uniformly $C^2$ and hyperbolic in the leaves. Therefore the behavior of most of the foliated geodesics is described by the $u$-Gibbs states of the geodesic flow:

\begin{corol}
Let $M$ be a compact Riemannian manifold with a foliation $\cF$, such that the metric is $C^3$ and its restriction to every leaf has negative sectional curvature. Denote by $X$ the generator of the foliated geodesic flow on $T^1\cF$.

1) Let $\mu$ be an ergodic u-Gibbs state for $X$
with normal Lyapunov exponents which are all negative, then $\mu$ is an SRB-measure  for the foliated geodesic flow.

2) If all ergodic u-Gibbs states for $X$
have all their normal Lyapunov exponents negative, then there are only a finite number of ergodic u-Gibbs states, each being an SRB-measure
and the union of their basins of attraction has full Lebesgue measure in $T^1\cF$.

\end{corol}

It is natural to consider the projection $p_*$ on the manifold of these $u$-Gibbs measures defined in the unit tangent bundle of the foliation.  One gets in this way a well defined class of measures absolutely continuous with respect to Lebesgue inside the leaves, and with strictly positive density (so that their support consists of entire leaves, see Theorem \ref{variable_curvature}).

For case 2 in Corollary 1, we define in section 6 the {\em visibility function}
$f_i:M \longrightarrow [0,1]$ which associates to a point $p$ of $M$ the
spherical Lebesgue measure in $T^1_p(M)$ of the set of directions at $p$ that correspond to geodesics through $p$ whose forward statistics tends to the ergodic u-Gibbs state $\mu_i$. We show:

\begin{theo}
\label{visibility_functions} Under the hypotheses of  part 2 of Corollary 1, the visibility function $f_i$ is continuous in $M$ and the value of $f_i$ is $1$ on the support  of the measure $p_*\mu_i$.
\end{theo}

If the laminations we  consider  are not just by negatively curved leaves but by hyperbolic leaves (that is, constant $-1$ sectional curvatures) then the  projection of $u$-Gibbs states to $M$  are precisely the harmonic measure introduced by L. Garnett in \cite{Garnett}. In this case the visibility function are {\em harmonic}, in the sense that their foliated Laplacian is zero.

We consider in section 7 smooth foliations by hyperbolic leaves which are transversely conformal.
In this setting, Deroin and Kleptsyn in \cite{Deroin-Kleptsyn} proved that, if the foliation does not admit any transverse invariant measure, then the foliation has finitely minimal sets, each of them carries exactly one harmonic measure, and the normal Lyapunov exponent of each of these measures is negative for the Brownian motion in the leaves.
We prove that the normal Lyapunov exponent of the corresponding $u$-Gibbs state coincides with the one of the Brownian motion, and therefore is negative, giving:

\begin{theo}
\label{no_invariant_measure}
Let $M$ be a compact manifold with a  transversly conformal foliation
$\cF$ of class $C^1$, with no transverse invariant measure,
endowed with a leafwise $C^3$ hyperbolic metric   and denote by $X$ the generator of the foliated geodesic flow on the unit tangent bundle $T^1\cF$ to the foliation. Then:

1) Every ergodic u-Gibbs state for $X$
has negative normal Lyapunov exponent.

2)  There are only a finite number of ergodic u-Gibbs states, each of them is an
SRB-measure and the union of their basins of attraction has full Lebesgue measure on $T^1\cF$.
\end{theo}

 Our initial interest in such foliations comes from the study of holomorphic foliations on complex surfaces. In some cases, one knows that the leaves are conformally hyperbolic surfaces but the hyperbolic metric can \emph{a priori} depend only continuously on the leaves \cite{Candel2}. This explains why some of our statements do not require transversal smoothness.

These results may be seen as generalization of \cite{Bonatti-GomezMont} where the $2$ first authors proved the existence of a unique u-Gibbs state for foliatons obtained as suspension of the fundamental group of a Riemann surface over finite volume hyperbolic surfaces, leading to the same result for some Riccati equations.

In sections 2 to 4 we deal with laminations,  in sections 5 and 6 we deal with foliations and in this last section  the foliation has leaves
of negative curvature.
In section 7 we deal with transversely conformal foliations with
hyperbolic leaves.

%% file: Sect2_sep2015.tex
\section{Laminated hyperbolicity}
\subsection{Basic definition of foliations and laminations}

\begin{defi}
A {\em lamination} or {\em foliated space} $(M,\cF)$ is a metrizable space $M$ together with an open covering $\{U_\alpha\}$
and homeomorphisms $\varphi_\alpha: \RR^n\times T_\alpha\to U_\alpha$ such that:
\begin{itemize}
\item[$\cdot$] $T_\alpha$ is a locally compact space.
\item[$\cdot$] Whenever $U_\alpha\cap U_\beta\neq \emptyset$, the change of coordinates $\varphi_\beta^{-1}\circ \varphi_\alpha$
has the form $\varphi_\beta^{-1}\circ \varphi_\alpha(x,t)=(g_{\alpha\beta}(x,t), h_{\alpha\beta}(t))$.
\item[$\cdot$] For fixed $t$ in $T_\alpha$, the map $g_{\alpha\beta}$ is differentiable.
\end{itemize}

A lamination is a {\em foliation} (of class $C^r$) if $T_\alpha=\RR^m$ and the maps $\varphi_\beta^{-1}\circ \varphi_\alpha$ are differentiable (of class $C^r$), in which case the space $M$ is a manifold of dimension $n+m$.
\end{defi}

The $U_\alpha \simeq \RR^n\times T_\alpha$ are called {\em foliated charts}. As usual, the notation for a foliated chart we will omit the
homeomorphism $\varphi_\alpha$. In a foliated chart, the sets of the form $\RR^n\times \{t\}$ are called {\em plaques}. Maximal connected manifolds that can be obtained by glueing plaques via de $g_{\alpha\beta}$ are called {\em leaves}.

\begin{defi}
A {\em Riemannian metric} on a lamination $(M,\cF)$ (of class $C^r$) is a choice of Riemannian metric on each leaf that, when seen in a foliated chart $\RR^n\times T_\alpha$,
has a continuous variation (in the $C^r$ topology) with respect to $t\in T_\alpha$. It induces on every leaf a distance  function $d_\cF$.
\end{defi}

In a foliation, a Riemannian metric can be obtained by restricting to the leaves a Riemannian metric in the ambient manifold $M$.
In a lamination, Riemannian metrics can be obtained by defining them locally and using partitions of unity.

\begin{defi}
Given a lamination $(M,\cF)$ with a Riemannian metric, its tangent bundle $T\cF$ is the lamination whose leaves are the  tangent spaces to the leaves of $\cF$,
where the laminated structure is defined in the natural way. The unit tangent bundle $T^1\cF \subset T\cF$ is formed by those vectors of norm 1.
\end{defi}

\subsection{Definition of laminated hyperbolicity}

For the rest of section 2, let $M$ be a compact metric space endowed with a lamination $\cF$,a Riemannian metric    and  $X$ a continuous vector field on $M$
tangent  to the plaques of $\cF$.

\begin{defi}\label{d.unifCr} The vector field will be called \emph{uniformly $C^r$ in the plaques} if
its expression in any laminated chart of $\cF$ induces a $C^r$ vector field in each plaque of
$\cF$, varying continuously with the plaque for the $C^r$-topology.
\end{defi}
\begin{exam}
Let $K\subset \RR$ be the compact set $\{0\}\cup\{\frac 1n\}_{n\in\NN^*}$ and $M= S^1\times K$ (where $S^1$ is the
circle $\RR/\ZZ$) trivialy laminated by the circles $S^1\times \{t\}$. We consider the vector field $X$ tangent to the leaves
and whose expression in $S^1\times \{\frac 1n\}$ is
$$X(t, \frac 1n)=\frac 1n\sin(2\pi nt)\frac{\partial}{\partial t}$$
and $X$ vanishes on $S^1\times \{0\}$.
Then $X$ is a continuous vector field, smooth in every leaf but not uniformly $C^1$ in the plaques.
\end{exam}

Assume that $X$ is uniformly $C^r$ in the plaques. Let $\{X_t\}_{t\in\RR}$ denote the flow of $X$.
For every $t$, $X_t$ is a homeomorphism of $M$ which is uniformly $C^r$ in the plaques. This allows us
to consider the differential $D_\cF(X_t)$ along the plaques, which defines a map on the tangent bundle to the
lamination $T\cF$ which is
uniformly $C^{r-1}$ along the plaques.

\begin{defi}\label{d.F-hyp} Let $M$ be a compact metric space, $\cF$ a   $C^r$-lamination, $r\geq 1$ on $M$
with a Riemannian metric on $T\cF$.
Let $X$ be a non-zero vector field tangent to $\cF$ which is uniformly $C^1$ in the plaques. We say that
$X$ is \emph{laminated hyperbolic}, or shortly \emph{$\cF$-hyperbolic}, if   $T\cF$  splits  in $D_\cF(X_t)$-invariant bundles
$$T\cF= E^s\oplus \RR X\oplus  E^u$$
such that there is  $T_0>0$ and $\lambda>1$ such that, for every $x\in M$ and any vectors $u\in E^s(x)$
and $v\in E^u(x)$ and any $T\geq T_0$ one has
$$
\|D_\cF(X_{T})(u)\|\leq \lambda^{-T}\|u\|\hskip 1cm
\mbox{ and } \hskip 1cm \|D_\cF(X_{T})(v)\|\leq \lambda^{T}\|v\|.
$$
\end{defi}

\subsection{  Stable and unstable laminations}

Exactly as in the usual hyperbolic case, the stable and unstable bundles of a laminated hyperbolic
vector field are always continuous.

Given any $x\in M$ we call the \emph{$\cF$-stable manifold of $x$}, and we denote it by $W^s_\cF(x)$,
the set of points $y$ in the leaf  $\cF_x$ through $x$ such that $\lim_{t\to +\infty} d_\cF(X_t(x) ,X_t(y))=0$.
In the same way the $\cF$-unstable manifold $W^u_\cF(x)$ is the set of $z\in \cF_x$ such that
$\lim_{t\to -\infty} d_\cF(X_t(x), X_t(z))=0$.

The local laminated stable and unstable manifolds of radius $\delta$, denoted by $W^s_{\cF,\delta}(x)$
and $W^u_{\cF,\delta}(x)$,  are the subsets of the laminated  stable and unstable manifolds whose positive
(resp negative) iterates remain at distance (in the leaf) less than $\delta$.

The proof of the existence of stable and unstable foliations (\cite{Shub}) can be directly
adapted to this case leading to the following statement:

\begin{theor}\label{t.invariant-manifolds} Let $(M,\cF)$ be a $C^r$-lamination on a compact
space and $X$  an $\cF$-hyperbolic
vector field which is uniformly $C^r$ along the plaques, $r\geq 1$, then there is $\delta>0$
such that the local laminated stable manifold  $W^u_{\cF,\delta}(x)$ is a $C^r$-disc centered at
$x$ in $\cF_x$ of the same dimension as $E^s$ and tangent to $E^s$.  The family of disks
$\{W^s_{\cF,\delta}(x)\}_{x\in M}$ is continuous for the $C^r$-topology.
The laminated local unstable manifolds $\{W^u_{\cF,\delta}(x)\}_{x\in M}$
form a
family of $C^r$-discs tangent to $E^u$ and varying continuously for the $C^r$ topology.
The global stable manifolds satify:
$$W^s_{\cF ,X}(x)= \bigcup_{t>0} X_{-t} W^s_{\cF,\delta}(X_t(x)).$$
The global $\cF$-stable manifolds are $C^r$-inmersed manifolds which fit into a $C^0$-lamination of $M$
by $C^r$-plaques, sublaminating the plaques of $\cF$ and varying $C^r$-continuously:
$\cW^s_{\cF,X}$ and $\cW^u_{\cF,X}$.
\end{theor}
\begin{proof}[Sketch of proof]
The classical proof of the existence of invariant manifold in the theory of hyperbolic diffeomorphisms consists in considering the space of $C^1$-disc-families $\{D^u_x\}_{x\in M}$, of a given radius and tangent to an invariant cone field
around the unstable bundle (in particular tranverse to the stable bundle), and depending continously with $x$ for the $C^1$-topology.
Then one considers the natural action of the diffeomorphism on one such family of discs: the discs become larger (since the dynamics expands the
unstable direction) and one shrinks them in order to recover a family of discs with the same properties.  The transverse contraction
implies that the action of the space of all such families is contracting, leading to a unique fixed family of
discs which is the announced family of local unstable manifolds.  This family depends continuously on the dynamics for the $C^1$-topology.
Applying the same argument to the differential of the diffeomorphism, one gets that the invariant manifolds are indeed $C^r$, and that they depend continouly
on the point and on the diffeomorphism for the $C^r$-topology.

Here we consider the families of discs  $D^u_x$ contained in the leaves, of a given radius,
and tangent to a conefield in the leaves around the unstable bundle, and depending continously on $x\in M$ for the $C^1$-topology.
One defines the action on the set of such families of discs exactly
as in the classical way. This action is still a contraction, leading to a unique fixed point which is the announced family of
local unstable manifolds. The higher regularity is obtained with the same trick.
\end{proof}

%% file: Sect3_sep2015.tex
\section{Gibbs $u$-States For $\cF$-hyperbolic Vector Fields}
In this section we consider a compact lamination $(M,\cF)$ endowed with an  $\cF$-hyperbolic vector field $X$ which is uniformly leafwise $C^2$ and
a Riemannian metric.

\subsection{Distortion control}
As explained in the previous section, the foliated unstable leaves form a foliation by $C^2$ leaves depending
$C^2$-continuously from the point. Hence, the restriction of the  flow $X_T$ is also uniformly $C^2$.  Furthermore,
for $T$ large enough, $X_{-T}$ induces a uniform contraction in the $\cF$-unstable leaves. In particular,
given two points $x,y$ in the same $\cF$-unstable leaf, the foliated distance $d_\cF(X_{-nT}(x),X_{-nT}(y))$ is
decreasing exponentially (with the rate $\lambda^{-T}$ so that the sum $\sum_0^{+\infty} d_\cF(X_{-nT}(x),X_{-nT}(y))$
converges and is bounded  proportionally to $d_\cF(x,y)$.
Furthermore the function $\log Det D^uX_{-T}$ is uniformly $C^1$ along the $\cF$-unstable leaves,
where the \emph{unstable derivative  $D^uX_{-T}$} is the restriction of the derivative  $D_\cF X_{-T}$ to the unstable direction  $E^u$.

One deduces the following foliated version of the classical distortion lemma:
\begin{lemm} \label{lemma:distortion_control} There is $C>0$ such that for every $x,y\in M$ in the same $\cF$-unstable leaf and for every $T>0$ one has

$$\left| \log |  Det D^u X_{-T}(x)|-\log| Det D^u X_{-T}(y)|\right|<C d_\cF(x,y)$$
\end{lemm}
\begin{demo}
$$\begin{array}{rc}
\left| \log |  Det D^u X_{-n}(x)|-\log| Det D^u X_{-n}(y)|\right|&\leq \\
\sum_0^{n-1} \left| \log |  Det D^u X_{-1}(X_{-i}(x))|-\log| Det D^u X_{-1}(X_{-i}(y))|\right|& \leq \\
C_0 \sum_0^{n-1} d_\cF(X_{-i}(x) X_{-i}(y))&
\end{array}
$$

\noindent
where $C_0$ is a bound for the derivative  of $\log |Det D^u X_{-1}|$ along the $\cF$-unstable leaves, and the constant
$C$ is given by $C=C_0\cdot \sum_0^{+\infty}\lambda^{-i}$.
\end{demo}

\subsection{Absolute continuity of the $\cF$-stable and unstable foliations}

When $X$ is uniformly $C^2$ in the leaves,  an important consequence of the distortion control is the absolute
continuity in the leaves of $\cF$ of the $\cF$-stable and $\cF$-unstable foliations.

\begin{defi} \begin{enumerate}
              \item Let $M$ be a manifold and $\cG$ be a foliation on $M$. One says that $\cG$ is \emph{absolutely
continuous} with respect to Lebesgue measure in $M$ if given any foliated chart $U$ of $\cG$ and $\Si$ a cross section intersecting the plaques
$P_x$ of $\cG$ in
$U$ in exacly $1$ point
one has the following property:
Given any measurable set  $A$ in $U$ such that $A$ has positive Lebesgue measure, consider  the set $\Sigma_A\subset\Si$  of points $x$ in $\Si$ so that
the plaque $P_x$ through $x$  cuts $A$ in a positive $P_x$-Lebesgue  measure set. Then
$\Sigma_A$ has positive measure for the  Lesbesgue measure of $\Si$.

\item Assume  that  $\cF$ is a lamination or a foliation  of $M$ and that $\cG$ is a subfoliation of $\cF$.  One says that $\cG$
is \emph{absolutely continuous with respect to Lebesgue in the leaves of $\cF$} if, for any leaf $\cL$ of $\cF$ the restriction of $\cG$ to $\cL$ is a foliation which is
absolutely continuous with respect to Lebesgue in $\cL$.
             \end{enumerate}
\end{defi}

\begin{theor} When $X$ is uniformly $C^2$ in the leaves,  the $\cF$-stable and $\cF$-unstable foliated invariant manifolds form foliations
$\mathcal{W}^s$ and
$\mathcal{W}^u$ of $M$ which are absolutely continuous with respect to Lebesgue in the leaves of $\cF$.
\end{theor}

The proof is identical to the one for  hyperbolic systems, see \cite{Brin}.

\subsection{Definition of u-Gibbs states}\label{ss.udef}
The distortion control provides a map which associates to pair $(x,y)$ of points in the same leaf of $\mathcal{W}^u$ the
limit
$$\psi^u(x,y)=\lim_{T\to +\infty} \frac{Det D^u X_{-T}(y)}{Det D^u X_{-T}(x)} >0.$$ (See \cite[Theorem 11.8]{BDV}.)
Futhermore this map is uniformly Lipschitz along the unstable leaves for the distance  $d_{\mathcal{W}^u}$ along the unstable leaves.
Finally, the map $\psi^u$ depends continuously on the unstable leaves in the following sense: Let $U$ be a foliated chart of the $\cF$-unstable foliation. Then the map $\psi^u$ is continuous on the set of pairs $(x,y)\in U^2$ such that $x$ and $y$ belong to the same plaque.

We will denote $\psi^u_x=\psi^u(x,.)$.  Notice that, for $y\in \mathcal{W}^u_x$, one has $$\psi^u_y=\psi^u_y(x)\cdot \psi^u_x.$$

Given a compact domain  $D$ in an unstable leaf,  we denote by $\cM_D$ the probability measure on $D$ whose density with respect
to the unstable volume $dm^u$ is proportional to $\psi^u_x$ for $x\in D$. In other words, given $x\in D$ one gets
$$\cM_D= \frac {\psi^u_x} {\int_D \psi^u_x dm^u}\cdot (dm^u|_D)  .$$

\begin{rema}  If $U\subset V$ are two domains contained in an unstable leaf, then $\cM_U=\frac 1{\cM_V(U)}\cM_V|_U$.
The probability measure $\cM_D$ is the limit of the image by $X_T$ of the normalized volume on $X_{-T}(D)$:
$$\cM_D=\lim_{T\to +\infty} (X_T)_*\left(\frac1{m^u(X_{-T}(D)}(dm^u|_{X_{-T}(D)})\right)$$
\end{rema}

\begin{defi}
Let $\mu$ be a probability on $M$. One says that $\mu$ is a \emph{u-Gibbs state for the $(M,\cF)$-hyperbolic vector field $X$} if:
\begin{enumerate}
\item $\mu$ is invariant by the flow of $X$;
\item  for any foliated chart of the unstable foliation, the desintegration of $\mu$ along the plaques induces the
probability measure $\cM_P$ on $\mu$-almost every plaque $P$.
\end{enumerate}
\end{defi}
Let us state more precisely the second condition in the definition:  let $\varphi\colon U \to V\times
\Sigma \subset \RR^k\times \Sigma$ be a foliated chart of the  unstable foliation $\{W^u_\cF(x)\}_{x\in M}$.
Rokhlin's desintegration theorem (\cite{Rokhlin}) asserts that there is a probability measure $\theta$ on
$\Sigma$ and a family of probabilities $\{\mu_p\}_{p\in\Sigma}$ such that
$\mu_p(V\times \{p\})=1$ for $\theta$-almost every $p\in\Sigma$ such that for any continuous function $h$
with support on $V\times\Sigma$ one has

$$\int_M h d\mu = \int_\Sigma\left(\int_{U\times \{p\}} h d\mu_p\right)d\theta(p).$$
 $\mu$ is a u-Gibbs state iff $\mu_p=\cM_{V\times \{p\}}$ for $\theta$ almost every $p\in\Sigma$. In other words

$$\int_M h d\mu = \int_\Sigma\left(\int_{U\times \{p\}} h d\cM_{V\times\{p\}}\right)d\theta(p).$$

\subsection{Existence and properties of u-Gibbs states}

\begin{prop}\label{uGibb1} For any $\cF$-unstable leaf $\cL^u$  and for any measurable set   $E$ contained in $\cL^u$ having a positive $\cL^u$-Lebesgue measure $0<m^u(E)<+\infty$, every accumulation point of the sequence of probability measures

$$\mu_T := \frac{1}{T}\int_0^TX_{t*}\left(\frac{m^u_E}{m^u_E(E)}\right)dt   \hskip 3cm T \longrightarrow \infty $$
 is a u-Gibbs  state.
\end{prop}
\begin{dem}
As explained in the proof of \cite[Lemma 11.12]{BDV}, we can assume without loss of generality that $E$
is a domain inside an $\cF$-unstable leaf. Then, the same arguments as in the proof of Pesin and Sinai's theorem
\cite[Theorem 11.8]{BDV} yield the desired result.
\end{dem}

\begin{rema}
The conclusion of the Proposition extends to $C^2$-disks contained in a leaf $\cL$ of $\cF$ and transverse to the $\cF$-centerstable foliation $\cW$.  In order to see this, just project such a disk onto an unstable domain in some unstable leaf, along the $\cF$-center stable foliation. The disk and its projection have the same average.
\end{rema}

\begin{prop} The set  $\cG^u_{ibbs}$ of u-Gibbs states is a non-empty  closed convex subset of the space of $X$-invariant probabilities.
Given a u-Gibbs state, its ergodic components are u-Gibbs states.
\end{prop}
\begin{demo}The fact that $\cG^u_{ibbs}$ is non empty follows from Proposition~\ref{uGibb1}. The fact that $\cG^u_{ibbs}$ is convex is straightforward. The fact that $\cG^u_{ibbs}$ is closed comes from the fact that for any $u$-Gibbs state its desintegration along the leaves has prescribed density functions $\psi^u$, which  depend continuously on the leaves.

Finally, let us show that the ergodic components of a u-Gibbs state $\mu$ are $u$-Gibbs states. Recall that a point $x$ is {\em regular} if its positive and negative average along the orbit are defined and are equal for every continuous function.  In other words, the probabilities $$\mu_T(x)=\frac1T dt|_{X_{[0,T]}(x)}\ \ \  \hbox{ and } \ \ \ \mu_{-T}(x)= \frac1T dt|_{X_{[-T,0]}(x)}$$ converge to the same probability measure $\mu_x$.  The set of regular points has full $\mu$-measure. Given a measure $\nu$ one says that a point $x$ is regular for $\nu$ if it is regular and if $\mu_x=\nu$.

The ergodic decomposition of $\mu$ is given as follows: We consider the probability measure on the space of $X$-invariant ergodic probabilities $\cP(X)$ such that the measure of a subset $\cP$ of $\cP(X)$ is the $\mu$-measure   of the union of the regular points for the elements of $\cP$. We can intersect this set of points with any measurable set of $\mu$-measure $1$.

As $\mu$ is a u-Gibbs state, there is a set $\cE$ of full $\mu$ measure of regular points  for which the intersection of $\cE$ with the unstable leaf  has total Lebesgue measure. Note that for $x$, $y$ regular in the same unstable leaf, one has $\mu_x=\mu_y$. Hence the set of regular points for any probability in the ergodic decomposition  of $\mu$ contains a subsets of full Lebesgue measure for the corresponding unstable leaf.

That is, for any ergodic component $\nu$ of $\mu$, there is at least one unstable leaf $\cL^u$ in which almost every point is
regular for $\nu$.  That is the limit of  $\mu_T(x)$ for almost every $x\in\cL^u$ is $\nu$. One deduces that the limit of
Lebesgue measure restricted to a disc in $\cL^u$ is $\nu$. But such a limit is a u-Gibbs state according to Proposition~\ref{uGibb1}. So $\nu$ is a u-Gibbs state, ending the proof.
\end{demo}

\begin{defi} An $X$-invariant probability measure $\mu$ on $M$ is \emph{ absolutely continuous in the unstable direction}
if for every laminated box $U=V_\cF\times \Sigma$ of $\cW^{u}$ such that $\mu(U)>0$ the conditional measure of $\mu\vert_{U}$
with respect to the partition into unstable plaques $\{W^u_\cF\times\{y\}:y\in \Sigma\}$ are absolutely continuous
with respect to Lebesgue measure along the corresponding plaque, for $\mu$ almost every plaque.
\end{defi}

As a consequence of  Proposition~\ref{uGibb1} and the  Proposition 3.8, one gets:
\begin{coro} If an $X$-invariant probability $\mu$ is absolutely continuous in the unstable direction, then it is a u-Gibbs state.
\end{coro}

\subsection{Almost every orbit accumulates on $u$-Gibbs states}
We will now prove Theorem~\ref{uGibb2}, which says that the $X$-orbit of Lebesgue almost every $x\in M$ distributes
according to a $u$-Gibbs state.

\begin{demt}
The proof follows the line of argument of the proof of the similar statement (\cite[Theorem 11.16]{BDV})
where one is assuming partial hyperbolicity. In the present case, we do not have the hypothesis of an invariant transversal
bundle to the $\cF$-unstable foliation $\cW^u$ for which the flow $X$ is expanding less than in
$\cW^{cu}$. What we do have is
the    foliation $\cF$.

We first fix an unstable domain $D$, that is a disc contained in an $\cF$-unstable leaf $\cL$. We want to show that, for almost every point $x\in D$ every accumulation point of the measures $\mu_T(x)=\frac1T dt|_{X_{[0,T]} (x)}$ is a u-Gibbs state.  That is, the mass of $\mu_T(x)$, for large $T$, is well distributed along the unstable leaves according to the densities $\psi^u_x$.  We will now formalize this idea.

Let $U\subset M$ be a $\cW^u$-foliated chart endowed with coordinates so that $U=B\times Z$ where the
$B\times\{ z\}$, $z\in Z$ are unstable discs ; one chooses this foliated chart so that it is also a flow box for $X$: the set $U$ is of the form $X_{[0,1]}(B\times Y)$, where $\Sigma= B  \times Y$ is a cross section of $X$.  We require furthermore that the unstable discs $B\times \{y\}$, $y\in Y$ are unstable discs of the same volume.

In order to understand the measure $\mu_T(x)$ in $U$ we have to understand the returns of the $X$-orbit
of $x$ in $U$. Notice that each component of a return map of the disc $D$ on $\Sigma$ is at constant time and is contained in a plaque $B\times \{y\}$: the flow preserves the $\cF$-unstable leaves, $D$ is contained in a $\cF$-unstable  leaf, and $\Sigma$ is foliated by $\cF$-unstable plaques.

One considers a disc $A\subset B$ \emph{far from the boundary of $B$}: $A$ is a subdisc centered at the same point as $B$ but with a radius $10$ times smaller. Up to shrinking $Y$ one may assume that for every
$y\in Y$
the $A\times \{y\}$ have the same volume $m_{B\times\{y\}}(A\times \{y\})$ as $A$:
$$  m_{B\times\{y\}}(A\times \{y\})=m_B(A).$$

The proof consists in showing that for almost every point $x$ of $D$ every accumulation points of the measures $\mu_T(x)=\frac1T dt|_{X_{[0,T]} (x)}$ gives $A$ a weight smaller than  $K m_B(A)$ where  $K$ is a constant independent of $A$. That is expressed in the following proposition which is Proposition 11.17 in \cite{BDV}.

Always following \cite{BDV}, let $K_0=\frac{m_B(A)}{m_B(B)}C$, where $C$ is the distortion control bound given in Lemma \ref{lemma:distortion_control}.

\begin{prop}\label{prop1}
Let $D$ be any unstable domain contained in a leaf of $\cW^u$. There exists $c>0$
such that given any subdisc $A$ far from the boundary of $B$ there exists
$n_0$ such that for $T\geq n_0$
$$m_D(\{x\in D \ : \ \mu_T(x)X_{[0,1]}(A \times Y) \geq 10K_0\frac{m_B(A)}{m_B(B)}\})\leq e^{-cT}.$$
\end{prop}

In order to prove Proposition~\ref{prop1} one first shows:

\begin{lemm}\label{l.proportion} Let $1\leq k\leq n$ and $0 \leq n_1 < \ldots < n_k<n$ be fixed. Then the $m_D$-measure of the set
of points $x \in D$ which hit $A\times Y$ at times $n_1,\dots n_k$ is bounded by
$(K_0\frac{m_B(A)}{m_B(\bar B)})^k$.
\end{lemm}

Proposition~\ref{prop1} is  deduced from Lemma~\ref{l.proportion} by using Stirling's formula exactly  as in \cite[Lemma 11.20]{BDV}
Let us just give the idea of the proof of Lemma~\ref{l.proportion} (\cite[Lemma 11.19 ]{BDV})

\begin{demo}First assume (even if it is not realistic) that the returns of the plaques $A\times \{y\}$ in
$A\times Y$ are Markovian, that is, each time $X_t(A\times \{y\})\cap (A\times Y)\neq \emptyset$ it consists in a plaque of $A\times Y$. Then the distortion lemma implies that for each sequence $n_1,\ldots,n_i$  of return times of a plaque $A\times\{y\}$ to $A\times Y$, the  proportion of points having a return  in $A\times Y$ at time $n_{i+1}$ will be exactly $\cM_D(A)$ times the proportion of points for which $n_{i+1}-n_i$ is a return time to $A\times Y$ leading to the announced estimates.

\vskip 3mm
However, the returns are not Markovian and we need to deal with returns of $A\times \{y\}$ which are only partially crossing $A\times Y$. However, this partial returns correspond to returns of points very close (exponentially close) to the boundary of $A$, because of the uniform dilation of $A$ by the dynamics. \cite[Lemma 11.19]{BDV} shows that replacing $X_t$ by some large iterates so that the expansion will be very large in comparison with the ratio of the radius of $B$ and $A$, one can increase a little bit $A$ and $B$ in $\tilde A\subset B$ and $\tilde B$ so that for every return of $A$ in $B$ the return of $\tilde B$ covers the whole $\tilde B$.  This allows us to neglect the effect of the non complete return.
\end{demo}

Proposition \ref{prop1}, together with a Borel-Cantelli argument, proves that given
any disk $A\subset B$
far from the boundary of $B$, there exists a full $m_D$-measure subset of points $x \in D$ such that
$$\mu_T(x) (X_{[0,1]}(A\times Y)) \leq 10K_0\frac{m_B(A)}{m_B(B)}$$
for all but, at most, a finite length set in of $T\in [0,\infty)$.  Consequently,
the same upper bound holds for $\mu(X_{[0,1]}(A \times Y))$, for any accumulation point $\mu$
of the measures $\mu_T$.
Applying the conclusion to a countable generating family of disks far from the boundary of $B$, and then taking the intersection of the corresponding full measure subset of $D$, we get a full $m_D$-measure subset of points $x \in D$
for which
$$\mu(X_{[0,1]}(A \times Y)) \leq 10K_0 \frac{m_B(A)}{m_B(B)}$$
for any measurable set $A$ far from the boundary of $B$.

To complete the proof, cover the manifold by foliated boxes $X_{[0,1]}(B\times Y)$
such that the foliated boxes obtained by replacing each $X_{[0,1]}(B\times Y)$ by  $X_{[0,1]}(A  \times Y)$ --where $A$
is, as before, far from the boundary of $B$--
 still covers $M$.
If $\mu(X_{[0,1]}(B\times Y))=0$ there is nothing to prove. Otherwise, by the
previous paragraph,
$$\frac{\mu(X_{[0,1]}(A \times Y))}{\mu (X_{[0,1]}( B \times Y))} \leq
\frac{10K_0}{\mu (X_{[0,1]}( B \times Y))}
\frac{m_B(A)}{m_B(B)}
$$
for every disk $A \subset B$ far from the boundary. This implies that $\mu$ is a u-Gibbs state, as claimed.
\end{demt}

%% file: Sect4_sep2015.tex
 \section{Laminations by leaves of negative curvature}

\subsection{u-Gibbs states}

In this section, we consider laminations  $\cF$ whose leaves have dimension $n$, admit a metric with strictly negative curvature, and  the metric depends continuously on the plaques for the $C^r$-topology, $r\geq 3$.
We denote by $(T^1\cF,  \tilde\cF)$ the unit tangent bundle to   the foliation,   which is itself a space laminated by (2n-1)-dimensional leaves, each leaf being the unit tangent bundle of a leaf of $(M,\cF)$ which carries a natural laminated Riemannian metric coming from the Riemannian metric on $M$ and the unit spherical metric on
the tangent vectors to the lamination.
Let us call $X$ the vector field on $T^1\cF$ generating the laminated geodesic flow $X_t$, $\pi:T^1\cF \to M$ the canonical projection, and $\pi_*$ the map it induces between measures on $T^1\cF$ and measures on $M$.
The fact that the metric tensor varies from leaf to leaf in a continous way with respect to the $C^r$ topology implies that the flow $X_t$, which has continous dependence on the derivatives of the metric up to order two, is a continuous flow on the space $T^1\cF$.

\vskip3mm
Then the foliated geodesic flow defined on $T^1\cF$ is uniformly $C^2$ in the leaves and is foliated hyperbolic.
Then Theorem~\ref{uGibb2} implies that  at every point $x\in M$ and for Lebesgue almost every direction $v\in T^1_x\cF$ every accumulation point of the probabilities associated to the positive orbit of the geodesic flow is a $u$-Gibbs state.

\vskip3mm
 The set of probabilities $\cG^u_{ibbs}$  on $T^1\cF$ which are $u$-Gibbs states for $X$ is a closed convex subset of the set of probability measures invariant under the geodesic flow. For every $u$-Gibbs state, its decomposition in ergodic components is through ergodic $u$-Gibbs states (Proposition 3.8).

As a direct consequence, the projection on $M$ of these $u$-Gibbs states are the probability measures which describe the asymptotic behavior of almost all geodesics in each plaque. In the case of hyperbolic leaves, these measures coincide with the harmonic measures (\cite{Connell-Martinez}). However, in the case of variable curvature, these measures are \emph{a priori} no longer harmonic measures. These measures form a new class of  measures which are naturally associated to the lamination and are interesting for understanding the dynamics of the lamination $\cF$. They have also been studied by Alvarez in \cite{A3}, \cite{A1} and \cite{A2}, where their relations with other classes of measures --e.g. harmonic measures-- is clearly explained.

\begin{theor}\label{variable_curvature} Let $(M,\cF)$ be a $C^3$-lamination
with  a plaque Riemannian metric of negative sectional curvature and let
$\mu$ be a $u$-Gibbs state of the geodesic flow on $T^1\cF$.  Let $\pi\colon T^1\cF\to M$ be the natural projection.  Then
 $\nu:= \pi_*\mu$ is a probability measure on $M$  whose desintegration $\nu_k$ along the plaques of $\cF$ satisfies:
 \begin{itemize}
 \item
 $\{\nu_k\}$ are absolutely continuous with respect to the Lebesgue measure of the plaques of $\cF$.
 \item The density of $\nu_k$ with respect to Lebesgue on the plaques is strictly positive on the whole plaque, and its logarithm is uniformly Lipschitz along the plaques on the whole manifold.
 \end{itemize}

 In particular the support of $\nu$ is $\cF$ invariant (it consists of entire leaves).

\end{theor}
\begin{demo}Let $\cL_x$, $x\in M$  be a leaf of $\cF$ and consider a direction $v\in T^1_x\cL_x$.  Consider now the center-stable manifold $W^{cu}(v)$ of $v$ for the geodesic flow.  It consists of all the vectors in $T^1_x\cL_x$ whose corresponding geodesic tends to the same point of the boundary at infinity of the universal cover $\tilde \cL_x$.  We fix a lift $\tilde x$ of $x$ on $\tilde \cL_x$.

$W^{cu}(v)$ is either diffeomorphic to an Euclidian space (if $v$ does not belong to the stable manifold of a closed orbit) or to a cylinder. We denote by $\tilde W^{cu}(v)$ its universal cover and we fix a lift $\tilde v$ of $v$. The natural projection of $\tilde W^{cu}(v)$ on  $\tilde \cL_x$ sending $\tilde v$ on $\tilde x$ is a diffeomorphism. Hence we endow every center unstable plaque with the lifted metric.

The desintegration of $\mu$ along center-stable plaques has a prescribed positive density whose logarithm is uniformly Lipschitz along the plaques.  Form the invariance of the metric by the geodesic flow and the fact that the action of the geodesic flow on the centralstable plaque has a bounded Jacobian, one deduces that the desintegration of $\mu$ on centerstable plaques has a logarithm which is uniformly Lipschitz with a prescribed constant.

If a family of positive functions has their logarithm which is Lipschitz for a given constant, then its sum or average has the same property. One deduces that the projection of $\mu$ on $\tilde \cL_x$ and therefore on $\cL_x$ has a density with respect to the Lebesgue measure whose logarithm is uniformly Lipschitz. As a consequence, the density is bounded away from $0$ and $+\infty$ in all compact domains.
\end{demo}

\subsection{Laminations by hyperbolic leaves}

We will now take $M$ to be a compact metrizable space with a lamination $\cF$ by   hyperbolic manifolds of dimension $n$ (each leaf carries its   metric of sectional curvature -1), such that the variation of the metric, as well as its derivatives, is continuous in the transverse directions.
The unstable lamination in $(T^1\cF, \tilde\cF)$ will be called the unstable-horosphere lamination $\cW^u$.
The laminated geodesic flow is preserving the foliation $\cW^u$ (but not the horospheres themselves), being a uniform expansion on the horospheres.
By a {\em horocycle measure} we mean a measure on $T^1\cF$ such that its desintegration with respect to $\cW^u$ is a constant multiple of the volume of the horospheres. In this case, they coincide with the $u$-Gibbs states, since in fact the unstable Jacobian equals $n-1$ so $\psi^u(x,y)\equiv 1$.

\begin{theor}\label{teoMatilde} (\cite{Bakhtin-Martinez}, \cite{Connell-Martinez})
Let $(M,\cF)$ be a  compact lamination by hyperbolic manifolds,  $\pi_*$ induces a bijective correspondence between horocycle measures on $T^1\cF$ which are invariant under the laminated geodesic flow and harmonic measures on $(M,\cF)$.
\end{theor}

\begin{coro}
Every Gibbs u-state on $T^1\cF$ projects to a harmonic measure on $M$.
\end{coro}

%% file: Sect5_sep2015.tex
\section{Normal Lyapunov exponents and  SRB-measures}

Let $\cF$ be a smooth foliation on a compact manifold $M$ and $X$  a uniformly $C^r$ vector field tangent the leaves, $r\geq2$.
Let $N_\cF$ be the normal bundle of $\cF$, that is, the quotient of the tangent bundle of $M$ by the tangent bundle to the foliation.  To each path $\gamma$ in a leaf, one may associate the linear holonomy from $N_\cF(\gamma(0))\to N_\cF(\gamma(1))$.
In this way, one defines a flow $\cX$ on $N_\cF$ as follows:
\begin{itemize}
\item $\cX$ is a linear cocycle over $X$, that is, $\cX$ leaves invariant the bundle structure over $M$,  the action on the leaves is linear and  $\cX$ projects on $X$ ;
\item $\cX_t|N_\cF(x)\colon N_\cF(x)\to N_\cF(X_t(x))$ is the linear holonomy of $\cF$ over the path $X_{[0,t]}(x)$
\end{itemize}

For time 1 we obtain a linear cocycle over a continuous dynamical system on a compact manifold, and we may apply
Oseledets' theorem (\cite{Katok}).  Hence for every $X$-invariant ergodic probability there is an associated
measurable $\cX$-invariant decomposition of the normal bundle in Lyapunov spaces on which the rate of
exponential growth of the vectors is well defined and constant and called the
\emph{normal Lyapunov exponents} of the ergodic measure.

\vskip3mm

We will now give a proof of Theorem \ref{main_theorem}. The first statement follows form classical results due to Pesin and Pugh and Shub, together with Oseledets' theorem.

\begin{demt2}
\begin{enumerate}
\item The flow $X_t$ is a smooth flow on $ M$ whose Lyapunov exponents for the measure $\mu$ are $\lambda^c=0$ (corresponding to the direction of the flow), $\lambda_j^s<0$, $\lambda_j^u>0$ (corresponding to the directions of the stable and unstable tangential directions to $\cF$, respectively) and the normal exponents $\lambda_j$ which are negative by hypothesis.  Oseledets' theorem   gives us a set $R$ (the {\em regular} set) that has full measure with respect to $\mu$, and over which there is an invariant measurable splitting
\begin{equation}
\label{Oseledets_splitting}
T_R M= E^{s}\oplus E^c \oplus E^u \oplus E^\pitchfork
\end{equation}
such that for every $x\in R$
\begin{equation}
\lim_{t\to \infty}\frac{1}{t} \log ||D_xX_t(v)||=-,0,+
\end{equation}
if $v\in E^j$, for $j=s,c,u$.

Pesin Stable Manifold Theorem (see \cite{Pesin} or \cite{Fathi-Herman}) tells us that for $\mu$-almost every $x\in R$ the local stable manifold at $x$
\begin{equation*}
W^s(x)=\{y\in  M:\ \limsup _{t\to\infty}\frac{1}{t}\log d(X_t(x),X_t(y))\}
\end{equation*}
is an embedded smooth   disk. Furthermore, Pesin proved that this foliation is absolutely continuous, see \cite{Pesin}.

Under this circumstances, there is a well known theorem due to Pugh and Shub (see \cite[Theorem 3]{Pugh-Shub}) that guarantees that the measure $\mu$ is an SRB measure for the flow $X_t$.

\item The arguments on the proof of this theorem are contained in \cite{BV}, although they are stated
for and used in a different context -- that of partial hyperbolicity with an additional assumption called
``mostly contracting center direction". The strong similarity of the ``mostly contracting" condition in the partially
hyperbolic setting and negativity of normal Lyapunov exponents in the setting of foliated hyperbolicity,
allows us to follow \cite{BV} closely.

More precisely, the difficulty in proving the finiteness of ergodic attractors stems from the fact that
the Pesin stable manifolds (tangent to the stable direction and the transverse direction) vary in a measurable way.
In particular, we have no {\it a priori} lower bound on their size. This problem is handled in \cite[Proposition 2.1]{BV},
where the absolute continuity of the Pesin stable ``foliation" is also established. This in turn implies, as a consequence of
\cite[Corollary 2.2]{BV} and Hopf's argument \cite{Brin}, that accessibility classes are open, which is \cite[Lemma 2.7]{BV}.
The finiteness of $u$-Gibbs states follows, as in \cite[Lemma 2.9]{BV}.
\end{enumerate}

\end{demt2}

%% file: Sect6_sep2015.tex
\section{SRB measures for the foliated geodesic flow: visibility functions of the attractors}

In this section we will consider a smooth foliation $\cF$ of the compact manifold $M$ and
a $C^3$ Riemannian metric on $M$ which induces negative curvature on all the leaves of $\cF$.
We denote by $\tilde\cF$ the foliation on the unit tangent bundle of the foliation whose leaves are the unit tangent bundles of the leaves.
The geodesic flow $X_t$ leaves the foliation $\tilde\cF$ invariant.

We assume here that all the transverse Lyapunov exponents
of the foliated geodesic flow are
negative for all ergodic u-Gibbs states.  In other words, one assume   the hypotheses of part 2 in
Theorem \ref{main_theorem}.
As a corollary of Theorems~\ref{uGibb2} and 2  one gets:

\begin{coro}\label{proportion} Given any point $x\in M$, for Lebesgue almost all unit vector $v\in T^1_x\cF$,
the corresponding orbit of the foliated geodesic flow belongs to the basin $\cB(\mu_i)$  of one of the SRB measures $\mu_i$.
\end{coro}
\begin{demo}Consider a plaque $\Delta$ of $\cF$ containing the point $x$ and consider the unit tangent bundle $T^1\Delta$.
According to Theorem \ref{main_theorem}, for almost every vector $v$ in $T_1\Delta$ there is $i$ so that
the orbit of $X_t$ through $v$ belongs to the basin of $\mu_i$.  Therefore, by Fubini theorem, there is a point $y\in\De$ so that the set
$T^1_y\Delta\cap \bigcup_i\cB(\mu_i)$ has total Lebesgue measure.
Now each $\cB_i$ is invariant under the $\cF$-centerstable foliation. Thus the image of $T^1_y\Delta\cap \bigcup_i\cB(\mu_i)$ by the
centerstable holonomy from $T^1_y\Delta$ to $T^1_x\Delta$ is contained in $T^1_x\Delta\cap \bigcup_i\cB(\mu_i)$.  As the center stable
foliation is absolutely continuous in the leaves of $\cF$, one gets that $T^1_x\Delta\cap \bigcup_i\cB(\mu_i)$ has total Lebesgue measure,
as announced.
\end{demo}

This property makes it natural to define
\begin{defi}
We call  \emph{visibility functions} of the measure $\mu_i$ and we denote by $f_i\colon M\to [0,1]$ the function defined by
$$f_i(x)=Leb_x( T^1_x\cF)\cap\cB(\mu_i)$$
where $Leb_x$ is the spherical measure on $T^1_x\cF$.
\end{defi}

Then corollary~\ref{proportion} implies that
$\sum_i f_i(x)=1$ for every $x\in M$.

\begin{rema} $\cB(\mu_i)$ is invariant for the center-stable foliation inside the leaves of $\tilde\cF$.
\end{rema}
According to the above remark, the functions $f_i$ vary smoothly  along the leaves, and the set of point for which
$f_i(x)>0$ is saturated for the foliation $\cF$. More precisely, each basin corresponds to some subset of the boundary
at infinity of the universal cover of the leaf, and the $f_i$ is the mass of the vectors corresponding to this set at infinity.

We are now in a position to prove Theorem \ref{visibility_functions}.

\begin{demt3} Every orbit in the basin of an SRB measure has well defined transverse  Lyapunov exponents for positive
iterations, and these Lyapunov exponents are negative. Hence the orbits have a transverse Pesin stable manifold and
this family of stable manifolds is absolutely continuous with respect to Lebesgue measure. One deduces that the functions
$f_i$ are lower semicontiuous.  As the sum of the $f_i$ is $1$, the semi continuity of the $f_i$ implies the continuity.
\end{demt3}

%% file: Sect7_sep2015.tex
\section{Transversely Conformal   Foliations with a $C^2$ $\cF$-hyperbolic vector field }

\subsection{ Negativity of the Lyapunov Exponent}

Throughout this section, let $(M,\cF)$ be a compact manifold with a smooth transversely conformal foliation
and $X$
a $C^2$ $\cF$-hyperbolic vector field. In this case, the
normal Lyapunov exponents are all equal.

We will make the additional assumption that the foliation is endowed with a Riemannian metric along the
leaves that has a H\"older
variation in the transverse direction.

Under the above hypothesis, Theorem \ref{main_theorem} says that if $\mu$ is an
ergodic u-Gibbs state for $X$ and its normal Lyapunov exponent is negative, then $\mu$ is an SRB-measure.

A special case of this is the following:

\begin{coro}
Let $(M,\cF)$ be a compact manifold with a smooth transversely conformal foliation, and suppose that $M$
has a $C^2$-Riemannian
metric that induces negative curvature on the leaves. Denote by
$X$ the generator of the foliated geodesic flow on $T^1\cF$ and let $\mu $ be an ergodic $u$-Gibbs state
for $X$. If
the Lyapunov exponent of the normal bundle of $\cF$ along $X$ is negative, then $\mu$ is an SRB-measure for
the foliated geodesic flow.
\end{coro}

In the special case where the leaves have not only negative curvature but constant curvature -1, this gives
rise to Theorem 4, the proof of which is the object of the present section.

Let us fix a transverse metric compatible with the conformal structure on trans\-ver\-sals; $|\cdot|$ is
its
associated norm.
 If $\gamma:[0,+\infty)\to M$ is a path lying on a single leaf,
let $h_{\gamma_{|[0,t]}}$ be the holonomy transformation corresponding to the path $\gamma$ between times
0 and $t$.
The {\em Lyapunov exponent} corresponding to $\gamma$ is the limit
\begin{equation}
\label{Lyapunov_exponent}
\lambda (\gamma) = \lim_{t\to +\infty} \frac{1}{t} \log |Dh_{\gamma_{|[0,t]}}|,
\end{equation}
if it exists. (When writing $Dh_{\gamma_{|[0,t]}}$ we omit the point $\gamma(0)$ where we are taking the
differential.)

The Birkhoff Ergodic Theorem and the Random Ergodic Theorem tell us that
for $\nu$-almost every $x\in M_0$ and almost every Brownian path $\gamma$ starting at x,
the above limit exists and is independent of both $x$ and  $\gamma$. It is called the Lyapunov exponent of
the
measure $\nu$, $\lambda(\nu)$.

The following result can be found in \cite{Deroin-Kleptsyn}:

\begin{theor} (Deroin-Kleptsyn)
Let $(M,\cF)$ be a compact minimal subset of a transversely conformal foliation having a Riemannian
metric
whose variation in the transverse direction is H\"older-continuous. Let $\nu$ be an $(M,\cF)$ harmonic measure.
If $(M,\cF)$ admits no holonomy invariant measures, then $\lambda(\nu)<0$.
\end{theor}

The requirement that the transverse variation of the metric be H\"older continuous is necessary,
see \cite{Deroin-Vernicos}.

Let $\mu^+$ be the ergodic horocycle measure on $T^1\cF$ which is invariant under the
geodesic flow and projects onto $\nu$.
It is an ergodic measure for the geodesic flow, and we can define
the Lyapunov exponent $\lambda(\mu^+)$ as the Lyapunov exponent of
the holonomy along a generic geodesic; namely
\begin{equation}
\label{Lyapunov_exponent_along_geodesics}
\lambda(\mu^+) = \lim_{t\to +\infty} \frac{1}{t} \log |Dh_{\gamma_{|[0,t]}}|,
\end{equation}
where $\gamma(s)=g_s(v)$ and $v\in T^1\cF$ is generic for $\mu^+$.

\begin{lemm}
With the above notations,
\begin{equation}
\label{equality_of_Lyapunov_exponents}
\lambda(\nu)=\lambda(\mu^+).
\end{equation}
\end{lemm}

The remainder of this subsection will be devoted to the proof of this result.

We will need expressions for the Lyapunov exponents $\lambda(\nu)$ and $\lambda(\mu^+)$ which are easier
to deal with.
We follow
\cite[Sec.3.1]{Deroin-Kleptsyn}.

Consider a holonomy-invariant local normal vector field $u$ in $T^1\cF$. Such a vector field can be defined
in every foliated chart,
and its   norm $|u|$ is globally defined up to a multiplicative constant on plaques. Therefore,
$\omega=d\log |u|$ is a globally defined foliated 1-form on $M$. (Here $d$ is the exterior derivative
in the leaf direction.) Together with the Riemannian metric on the leaves,
it determines $\tilde Y=\nabla \log|u|$, which is a vector field along the leaves of $T^1\cF$.

Applying the same argument to the foliation $\cF$, we obtain the vector field $Y=\nabla \log|u|$
on $M$ which is tangent to the leaves of $\cF$.

As in \cite{Deroin-Kleptsyn}, we can write the Lyapunov exponent of the ergodic harmonic measure $\nu$ as
\begin{equation}
\lambda(\nu)=\int_M \Delta \log |u|\, d\nu = \int_M \div Y\, d\nu,
\end{equation}
where $\div$ is the leafwise divergence.

Now consider a path $\gamma:[0,T]\to T^1\cF$ lying on a leaf, and let $h_\gamma$ be
the holonomy tranformation it determines.
The contraction or expansion of the holonomy along $\gamma$ is
\begin{equation*}
|Dh_\gamma|= e^{\int_\gamma \omega}.
\end{equation*}
Therefore, the Lyapunov exponent along a geodesic orbit is
\begin{equation}
\lambda (v)=\lim_{t\to +\infty} \frac{1}{t} \int_0^t \omega(g_s(v))\, ds.
\end{equation}
If $X$ is the vector field that directs the foliated geodesic flow on $T^1\cF$,
the Lyapunov exponent of the measure $\mu^+$ is
\begin{equation}
\lambda(\mu^+)=\int _{T^1\cF} \lambda (v)\, d\mu^+(v) = \int _{T^1\cF} \omega(X)\, d\mu^+.
\end{equation}

Desintegration of the harmonic measure $\nu$ on a flow box of $(M,\cF)$
yields a measure on its transversal and conditional measures
of the form $h\cdot vol$ on its plaques, where $h$ is a positive function which is harmonic in the leaf
direction and $vol$
is the volume on leaves. On two overlapping flow boxes, the corresponding harmonic densities $h$ and $h'$
coincide up to
multiplication by a positive function which is constant on plaques. Therefore, $\eta=d\log h$ is a
well-defined foliated
1-form, which is called the {\em modular form} of $\cF$.

The following result, which we will use to relate the exponents $\lambda(\nu)$ and $\lambda(\mu^+)$,
appears in
\cite[p.210]{Candel2}.

\begin{lemm}
\label{Lemma_Candel}
The modular form satisfies
\begin{equation*}
\int _M \div Y\, d\nu = -\int _M \eta(Y)\, d\nu
\end{equation*}
for every vector field $Y$ along the leaves with an integrable divergence.
\end{lemm}

Taking $Y=\log |u|$ as above, the term on the left is the Lyapunov exponent of the measure $\nu$. Equation
(\ref{equality_of_Lyapunov_exponents}) translates into
\begin{equation}
\label{equality_of_Lyapunov_exponents_2}
-\int_M \eta(Y)\, d\nu = \int_{T^1\cF} \langle \tilde Y,X \rangle \, d\mu^+.
\end{equation}
(Here $\langle\cdot,\cdot\rangle$ is the leafwise metric on $\cF$ in the term on the left and
the leafwise metric on $T^1\cF$ in the term on the right.)

Considering a partition of unity $\{\varphi_i\}$ subordinated to a foliated atlas on $M$ and the
corresponding
partition of unity $\{\varphi_i\circ\pi\}$ on $T^1\cF$. The proof of equation
\ref{equality_of_Lyapunov_exponents_2}
is reduced to a local computation. So let $E\simeq D\times T$ be a foliated chart on $M$, where $D$ is a
ball in hyperbolic space, and
$\pi^{-1}(E)\simeq D\times T\times S^{n-1}$ its preimage on $T^1\cF$. Notice that the transverse components
of $\nu$
and $\mu^+$ are equal, and therefore equation (\ref{equality_of_Lyapunov_exponents_2}) must actually be an
equation on
each plaque.  Namely, we have to prove that
\begin{equation}
\label{equality_of_Lyapunov_exponents_plaque_level}
 -\int_D \langle Z(x), Y(x)\rangle h(x)\, dx = \int_{D\times S^{n-1}} \langle X, \tilde Y\rangle\, d\mu^+_{D\times S^{n-1}},
\end{equation}
where $h$ is the harmonic density induced by $\nu$ in a plaque of $E$, $\mu^+_{D\times S^{n-1}}$ is
the conditional measure induced by $\mu^+$ on the corresponding plaque in $\pi^{-1}(E)$, and $
Z=\nabla\log h$.

Each plaque in $\pi^{-1}(E)$ is itself foliated by (local)
central-unstable leaves, parametrized by $\xi\in \partial D$. We
define a 1-form which is a foliated form for this central-unstable
foliation by the expression
\begin{equation*}
 \tilde\eta = d\log k(\cdot, \xi)
\end{equation*}
on the leaf corresponding to $\xi$, where $k$ is the Poisson kernel on the $D$. This is a globally defined
1-form
on $T^1\cF$.

\begin{lemm}
 \label{integral_wrt_modular_form}
\begin{equation}
 \int_{T^1\cF} \tilde\eta(\tilde Y)\, d\mu^+ = \int_M \eta(Y)\, d\nu.
\end{equation}
\end{lemm}

Equation (\ref{equality_of_Lyapunov_exponents_plaque_level}) follows  from this lemma, and the proof
of the lemma
is   straightforward. Nevertheless, we will write it in full detail.

\begin{dem}
It is enough to prove the lemma in a foliated chart $E\simeq D\times
T$ where $\pi^{-1}(E)\simeq D\times S^{n-1}\times T$. Since
desintegration of the measures $\nu$ in $E$ and $\mu^+$ in
$\pi^{-1}(E)$ yield the same measure on $T$, it is in fact enough to
prove it at plaque level.

Notice that $\tilde Y=\pi^*Y$, and therefore it is constant on the
fibers of $\pi$.

If $h$ is a positive harmonic function on $D$, we call $m_h$ the
measure on $S^{n-1} = \partial D$ which corresponds to $h$ via
Poisson's representation formula. The symbol $\nabla$ will always refer to the
gradient with respect to $x\in D$ with the leafwise metric.
\begin{eqnarray*}
\int_{S^{n-1}}\int_D \tilde\eta (\tilde Y) k(x, \xi)\, dx\, dm_h(\xi)
&=& \int_{S^{n-1}}\int_D \langle \nabla\log k(x,\xi),\nabla\log
|u|(x)\rangle k(x,\xi)\, dx\, dm_h(\xi)\\
&=& \int_D\int_{S^{n-1}} \left \langle \frac{\nabla k(x,\xi)}{k(x,\xi)},
\nabla\log |u|(x)\right \rangle k(x,\xi)\, dm_h(\xi)\, dx\\
&=& \int_D \left\langle \int_{S^{n-1}} \nabla k(x,\xi)dm_h(\xi), \nabla
|u|(x)\right\rangle\, dx\\
&=& \int_D \left\langle \nabla \int_{S^{n-1}} k(x,\xi)dm_h(\xi), \nabla
|u|(x)\right\rangle\, dx\\
&=& \int_D \langle \nabla h(x), \nabla\log |u|(x)\rangle\, dx\\
&=& \int_D \langle \nabla \log h(x), \nabla\log |u|(x)\rangle\, h(x)
dx.
\end{eqnarray*}
\end{dem}

Finally, we will complete the proof of  Theorem 4.

\begin{demt4}
Equation (\ref{equality_of_Lyapunov_exponents_plaque_level}) follows
directly from Lemma \ref{integral_wrt_modular_form} when we observe
that on the central-unstable leaf for the geodesic flow in $D\times
S^{n-1}\simeq T^1D$ corresponding to geodesics born at $\xi\in S^{n-1}$,
$X=-\nabla \log k(\cdot, \xi)$.
\end{demt4}

\subsection{The regular orbits}
We have seen that almost every  positive orbit of the foliated geodesic flow follows one of the SRB
measures corresponding to the harmonic measures. That is, if $\nu_1,\ldots,\nu_k$ are the ergodic harmonic
measures
on $(M,\cF)$, they give rise to $\mu^+_1,\ldots,\mu^+_k$
which are SRB measures for the positive orbits of the foliated geodesic flow. We can follow the same line
of reasoning for the negative orbits of the geodesic flow, thus obtaining measures $\mu^-_1,\ldots,\mu^-_k$
which also project onto the $\nu_i$. In general the the measures $\mu^+_i$ and $\mu_i^-$ do not coincide.
They do so if and only if the harmonic measure $\nu_i$ is {\em totally invariant}, which means that
in each foliated chart it can be expressed as a transverse holonomy-invariant measure times the leaf
volume. This is the content of the following proposition:

\begin{prop}
\label{holonomy_invariant}
If $\mu$ is a measure on $T^1\cF$ which is invariant by the geodesic flow and which has
constant density along both stable and unstable horospheres, then it projects on $M$ as a measure  whose
desintegration in the leaves is the volume in the leaves, implying that this measure is locally the product
of the volume in the leaves by a transverse invariant measure.
\end{prop}
\begin{demo} That is exactly the same proof as in \cite{Bonatti-GomezMont}, and can also be found in
\cite[Theorem 5.2]{Connell-Martinez}.
\end{demo}

Recall that, according to Birkhoff's theorem, almost every orbit
(for any invariant measure) is regular in the sense that the positive and negative averages exists and are
equal.

\begin{prop}Assume that $\cF$ does not admit any transverse invariant measure. Then the set of regular
point in $T^1\cF$ has Lebesgue measure equal to $0$.
\end{prop}

\begin{demo} For each $i$, let $\cB(\mu^+_i)$ the basin of attraction of $\mu^+_i$, and
$\cB(\mu^-_i)$ the set of points whose negative orbit is contained in the basin of attraction of
$\mu^-_i$. Assume that $E\subset \tilde M$ is a positive Lebesgue measure set of regular points of the
foliated geodesic flow.  Up to removing from $E$ a $0$ Lebesgue measure set, one may assume that almost
every $x\in E$ belongs to some $\cB(\mu^+_i)$ and to some
$\cB(\mu^-_j)$. As the positive an negative SRB measures are finitely many, this impies that there are
$i,j$ and a positive Lebesgue measure set $E_0\subset E$ included in $\cB(\mu^+_i)\cap \cB(\mu^-_j)$.
However,
as the points in $E$ (hence in $E_0$) are assumed to be regular, the positive and negative averages
coincide, that is, the integral of every continuous function for $\mu^+_i$ or $\mu^-_j$ are equal, meaning
that $\mu^+_i=\mu^-_j$.  Since $\mu_i^+$ and $\mu^-_j$ project onto $\nu_i$ and $\nu_j$ respectively,
this immediately means that $i=j$.
Therefore $\mu^+_i=\mu^-_i$ is a measure which is invariant by the geodesic
flow, and it has constant density along both stable and unstable horospheres. This,
together with Proposition \ref{holonomy_invariant}, gives a contradiction.
\end{demo}

 \begin{coro} If $\cF$ has no transverse holonomy-invariant measure,
 each of the measures $\mu^+_i$ is singular
 with respect to Lebesgue measure on $T^1\cF$.
\end{coro}

For $x\in M$ and $v\in T^1_x\cF$, let $\ell_{(x,v)}(T)$ be the geodesic segment in the
leaf $\mathcal{L}_x$ centered at $x$ and
directed by $v$. Let $m_{(x,v)}(T)$ the normalized length measure of $\ell_{(x,v)}(T)$, which
is a probability measure on $M$.

\begin{prop}
 If $\cF$ admits only one harmonic measure $\nu$ which is not totally invariant, then for Lebesgue
 almost every $x\in M$ and Lebesgue almost every $v\in T^1_x\cF$, the measures $m_{(x,v)}(T)$
 converge to $\nu$ as $T\to +\infty$.
\end{prop}

\begin{demo}
If $(M,\cF)$ admits only one harmonic measure $\nu$ which is not totally invariant, Lebesgue
almost every orbit of the geodesic flow in $T^1\cF$ distributes according to $\mu^+$ in the future
and $\mu^-$ in the past. Such orbits project onto geodesics with the desired property.
\end{demo}

\begin{exam} Riccati Equations: Let $S$ be a compact hyperbolic Riemann surface and $\rho:\pi_1(S) \longrightarrow PSL(2,\CC)$
a representation of its fundamental group into the group of Moebius transformations of the Riemann sphere $\PP^1$.
Let $p:M_\rho \longrightarrow S$ be the (ruled) algebraic surface obtained by the suspension of $\rho$ (i.e. $\tilde S \times_{\pi_1(S)} \PP^1$
with the holomorphic foliation $\cF_\rho$ transversal to the ruling obtained by taking the quotient of the trivial foliation
$\cup_{t\in \PP^1}\tilde S \times \{t\}$). It is a Ricatti equation (\cite{Bonatti-GomezMont},\cite{Elifalet}).
Let $T^1S\rightarrow S$ be the unit tangent bundle to $S$ with the induced morphism of fundamental groups
$ p_*:\pi_1(T^1S) \longrightarrow \pi_1(S)$.
The unit tangent bundle $T^1\cF$ is obtained by suspending the representation $\rho \circ p_*:\pi_1(T^1S) \longrightarrow PSL(2,\CC)$.
The geodesic flow and the foliated geodesic flow commute with the projection $\tilde p:T^1\cF \rightarrow T^1S$.
It is shown in \cite{Bonatti-GomezMont} that there are measurable sections $\sigma^\pm:T^1S \rightarrow T^1\cF$
commuting with the geodesic flows, such that $\mu^\pm:=\sigma^\pm(dLiouv_{T^1S})$ are the attractors and repellors (the u-Gibbs states)
of the foliated geodesic flow. If  $\rho$ is the representation of the universal cover of $S$, or another quasi-Fuchsian representation,
the sections are continuous, the limit sets of the discrete groups are circles or quasi-circles, and the support of the u-Gibbs states
are the 3-manifolds $\sigma^\pm( T^1S )$. We have in this case a north to south pole dynamics. If the representation is non-discrete,
so that the limit set of the representation is the Riemann sphere, the general orbit of $\rho(\pi_1(S))$ will be dense in $\PP^1$,
the general leaf of $\cF$ will be dense in $M_\rho$
and  the support of the u-Gibbs state is all of $T^1\cF$. In this last case one has a strange attractor/repellor dynamics.
\end{exam}